\documentclass[12pt]{article}
\usepackage{amsmath,amssymb,amsthm, comment}

\newtheorem{theorem}{Theorem}[section]
\newtheorem{thmy}{Theorem}

\newtheorem{lemma}[theorem]{Lemma}
\newtheorem{corollary}[theorem]{Corollary}
\newtheorem{conjecture}[theorem]{Conjecture}

\def\barr{\begin{array}}
\def\earr{\end{array}}

\title{A criterion for nilpotency of a finite group by the sum of element orders}
\author{Marius T\u arn\u auceanu}
\date{March 23, 2019}

\begin{document}

\maketitle

\begin{abstract}
Denote the sum of element orders in a finite group $G$ by $\psi(G)$ and let $C_n$ denote the cyclic group
of order $n$. In this paper, we prove that if $|G|=n$ and $\psi(G)>\frac{13}{21}\,\psi(C_n)$, then $G$ is
nilpotent. Moreover, we have $\psi(G)=\frac{13}{21}\,\psi(C_n)$ if and only if $n=6m$ with $(6,m)=1$ and
$G\cong S_3\times C_m$. Two interesting consequences of this result are also presented.
\end{abstract}

{\small
\noindent
{\bf MSC2000\,:} Primary 20D60; Secondary 20D15, 20F18.

\noindent
{\bf Key words\,:} group element orders, nilpotent groups.}

\section{Introduction}
Given a finite group $G$, we consider the functions
\begin{equation}
\psi(G)=\sum_{x\in G}o(x) \mbox{ and } \psi'(G)=\frac{\psi(G)}{\psi(C_{|G|})}\,,\nonumber
\end{equation}where $o(x)$ denotes the order of $x$. In \cite{1}, H. Amiri, S.M. Jafarian Amiri and I.M. Isaacs proved the following theorem:

\begin{thmy}
If $G$ is a finite group, then
\begin{equation}
\psi'(G)\leq 1,\nonumber
\end{equation}and we have equality if and only if $G$ is cyclic.
\end{thmy}In other words, the cyclic group $C_n$ is the unique group of order $n$ which attains the maximal value $1$ of $\psi'(G)$ among groups of order $n$.\newpage

Since then many authors have studied the function $\psi(G)$ and its relations with the structure of $G$ (see e.g. \cite{2}-\cite{10}). In the papers \cite{4} and \cite{10} M. Amiri and S.M. Jafarian Amiri, and, independently, R. Shen, G. Chen and C. Wu started the investigation of groups with the second largest value of the sum of element orders. M. Herzog, P. Longobardi and M. Maj \cite{6} determined the exact upper bound for $\psi(G)$ for non-cyclic groups of order $n$:

\begin{thmy}
If $G$ is a finite non-cylic group and $q$ is the least prime divisor of the order of $G$, then
\begin{equation}
\psi'(G)\leq f(q)=\frac{\left[(q^2-1)q+1\right](q+1)}{q^5+1}\,,\nonumber
\end{equation}and the equality holds if and only if $|G|=q^2m$ with $(m,q!)=1$ and $G\cong(C_q\times C_q)\times C_m$.
\end{thmy}Note that the above function $f$ is strictly decreasing on $[2,\infty)$ and consequently the largest value of $\psi'(G)$ is $f(2)=\frac{7}{11}$\,, which is attained for $G\cong(C_2\times C_2)\times C_m$ with $m$ odd. Also, for any prime $q$ we have
\begin{equation}
f(q)<\frac{1}{q-1}\,.
\end{equation}

By using the sum of element orders, several criteria for solvability of finite groups have been also determined (see e.g. \cite{5,7}). Recall here the following theorem of M. Baniasad Asad and B. Khosravi \cite{5}:

\begin{thmy}
If for a finite group $G$ we have
\begin{equation}
\psi'(G)>\frac{211}{1617}\,,\nonumber
\end{equation}then it is solvable.
\end{thmy}Note that the groups $G\cong A_5\times C_m$ with $(30,m)=1$ satisfy $\psi'(G)=\frac{211}{1617}\,$.
\vspace{0.5mm}

Finally, we recall a recent result of M. Herzog, P. Longobardi and M. Maj \cite{8}, which gives an exact upper bound for $\psi(G)$ for non-cyclic groups of order $2m$ with $m$ odd:

\begin{thmy}
If $G$ is a non-cyclic group of order $2m$ with $m$ odd, then
\begin{equation}
\psi'(G)\leq\frac{13}{21}\,,\nonumber
\end{equation}and the equality holds if and only if $G\cong S_3\times C_{m'}$ with $(6,m')=1$.
\end{thmy}

Our main result is the following theorem.

\begin{theorem}
If $G$ and a finite group and
\begin{equation}
\psi'(G)>\frac{13}{21}\,,\nonumber
\end{equation}then $G$ is nilpotent. Moreover, we have $\psi'(G)=\frac{13}{21}$ if and only if $G\cong S_3\times C_m$ with $(6,m)=1$.
\end{theorem}

Using Theorem 1.1 and Lemma 2.2, we are able to determine the largest four values of $\psi'$ and the groups for which they are attained.

\begin{corollary}
Let $G$ be a finite group satisfying $\psi'(G)>\frac{13}{21}$\,. Then $\psi'(G)\in\{\frac{27}{43}\,, \frac{7}{11}\,, 1\}$, and one of the following holds:
\begin{itemize}
\item[{\rm a)}] $G\cong Q_8\times C_m$, where $m$ is odd;
\item[{\rm b)}] $G\cong (C_2\times C_2)\times C_m$, where $m$ is odd;
\item[{\rm c)}] $G$ is cyclic.
\end{itemize}
\end{corollary}

In other words, $\frac{13}{21}$ is the fourth largest value of $\psi'$ on the class of finite groups.
\smallskip

A generalization of Theorem D can be also inferred from Theorem 1.1.

\begin{corollary}
If $G$ is a non-cyclic group of order $2^km$ with $m$ odd and $k\neq 2,3$, then
\begin{equation}
\psi'(G)\leq\frac{13}{21}\,,\nonumber
\end{equation}and the equality holds if and only if $G\cong S_3\times C_{m'}$ with $(6,m')=1$.
\end{corollary}

For the proof of Theorem 1.1, we need some preliminary results from the papers \cite{1} and \cite{6}.

\begin{lemma}
The following statements hold:
\begin{itemize}
\item[{\rm 1)}]{\rm ((\cite{6}, Lemma 2.2(3))} $\psi$ is multiplicative, that is if $G=A\times B$, where $A,B$ are subgroups of $G$ satisfying
$(|A|,|B|)=1$, then  $\psi(G)=\psi(A)\psi(B)$; note that this implies that $\psi'$ is also multiplicative;
\item[{\rm 2)}]{\rm ((\cite{6}, Lemma 2.9(1))} $\psi(C_{p^n})=\frac{p^{2n+1}+1}{p+1}$\,;
\item[{\rm 3)}]{\rm ((\cite{6}, Proof of Lemma 2.9(2))} If $n$ be a positive integer larger than $1$, with the largest prime
divisor $p$ and the smallest prime divisor $q$, then $\psi(C_n)\geq\frac{q}{p+1}n^2$\,;
\item[{\rm 4)}]{\rm ((\cite{1}, Corollary B)} If $P$ is a cyclic normal Sylow subgroup of $G$ then $\psi(G)\leq\psi(P)\psi(G/P)$, with equality if and only if $P$ is central in $G$;
\item[{\rm 5)}]{\rm ((\cite{6}, Lemma 2.2(5))} If $G=P\rtimes H$, where $P$ is a cyclic $p$-group, $|H|>1$ and $(p,|H|)=1$, then $\psi(G)=|P|\psi(H)+(\psi(P)-|P|)\psi(C_H(P))$.
\end{itemize}
\end{lemma}

Inspired by the above results, we came up with the following conjecture.

\begin{conjecture}
If $G$ is a finite group and
\begin{equation}
\psi'(G)>\frac{31}{77}\,,\nonumber
\end{equation}then $G$ is supersolvable. Moreover, we have $\psi'(G)=\frac{31}{77}$ if and only if $G\cong A_4\times C_m$ with $(6,m)=1$.
\end{conjecture}

\section{Proofs of the main results}

Throughout this section, given a finite group $G$ we will denote by $q$ and $p$ the smallest
and the largest prime divisor of $|G|$, respectively.

\begin{lemma}
Let $G$ be a finite group. If $\psi'(G)>\frac{13}{21}$ and $G$ is not a $2$-group, then it has a cyclic normal Sylow $r$-subgroup, where either $r=2$ or $r=p$.
\end{lemma}

\begin{proof}
If $G$ is cyclic, we are done. If $G$ is not cyclic, then the conditions $\psi'(G)>\frac{13}{21}$ and (1) imply $q=2$. Also, we have $p\geq 3$ since $G$ is not a $2$-group. By Lemma 1.4, 3), it follows that
\begin{equation}
\psi(G)>\frac{13}{21}\,\psi(C_{|G|})\geq\frac{13}{21}\,\frac{2}{p+1}\,|G|^2=\frac{26}{21(p+1)}\,|G|^2\nonumber
\end{equation}and so there exists $x\in G$ with $o(x)>\frac{26}{21(p+1)}\,|G|$, i.e.
\begin{equation}
[G:\langle x\rangle]<\frac{21(p+1)}{26}\,.\nonumber
\end{equation}

Suppose first that $p\geq 5$. Then $\frac{21(p+1)}{26}<p$ and thus $\langle x\rangle$ contains a cyclic Sylow $p$-subgroup $P$ of $G$. Since $\langle x\rangle\leq N_G(P)$, it follows that $P$ is normal in $G$, as desired.\newpage

Next assume that $p=3$. Then $[G:\langle x\rangle]<\frac{42}{13}$ and thus $[G:\langle x\rangle]\in\{2,3\}$. If $[G:\langle x\rangle]=2$, then $\langle x\rangle$ contains a cyclic normal Sylow $3$-subgroup of $G$, as above. If $[G:\langle x\rangle]=3$, then $\langle x\rangle$ contains a cyclic Sylow $2$-subgroup $Q$ of $G$, and we have $\langle x\rangle\leq N_G(Q)$. Therefore either $N_G(Q)=G$, i.e. $Q\vartriangleleft G$, or $N_G(Q)=\langle x\rangle$. In the latter case, assume that there exists $y\in G\setminus\langle x\rangle$ with $o(y)=\frac{|G|}{2}\,$. Then $\langle y\rangle$ will contain a cyclic normal Sylow $3$-subgroup of $G$, and we are done. Assume now that $o(y)\leq\frac{|G|}{3}\,$, for all $y\in G\setminus\langle x\rangle$, and put $|G|=2^a3^b$ with $a,b\geq 1$. Using Lemma 1.4, 1) and 2), one obtains
$$\psi(G)=\psi(\langle x\rangle)+\psi(G\setminus\langle x\rangle)\leq\psi(\langle x\rangle)+\frac{2}{9}\,|G|^2=$$
$$\hspace{1mm}=\frac{(2^{2a+1}+1)(3^{2b-1}+1)+2^{2a+3}3^{2b-1}}{12}\,,$$and consequently
$$\psi'(G)\leq\frac{(2^{2a+1}+1)(3^{2b-1}+1)+2^{2a+3}3^{2b-1}}{(2^{2a+1}+1)(3^{2b+1}+1)}\leq\frac{3^{2b-1}+1}{3^{2b+1}+1}\,+$$
$$\hspace{3mm}+\frac{4}{9}\,\frac{2^{2a+1}3^{2b+1}}{(2^{2a+1}+1)(3^{2b+1}+1)}<\frac{1}{7}+\frac{4}{9}=\frac{35}{63}<\frac{13}{21}\,,$$a contradiction.
\end{proof}

In the following lemma we determine all finite nilpotent groups $G$ satisfying $\psi'(G)>\frac{13}{21}\,$.

\begin{lemma}
Let $G$ be a finite group. If $\psi'(G)>\frac{13}{21}$ and $G$ is nilpotent, then one of the following holds:
\begin{itemize}
\item[{\rm a)}] $G\cong Q_8\times C_m$, where $m$ is odd;
\item[{\rm b)}] $G\cong (C_2\times C_2)\times C_m$, where $m$ is odd;
\item[{\rm c)}] $G$ is cyclic.
\end{itemize}
\end{lemma}

\begin{proof}
Suppose that $G$ is not cyclic. As in the proof of Lemma 2.1, we have again $q=2$. Let $n=p_1^{n_1}p_2^{n_2}\cdots p_k^{n_k}$ be the decomposition of $n$ as a product of prime factors, where $k\in\mathbb{N}^*$, $p_1=2$ and $p_1<p_2<\cdots<p_k$. Since $G$ is nilpotent, it can be written as the direct product of its Sylow $p_i$-subgroups
\begin{equation}
G\cong G_1\times G_2\times\cdots\times G_k.\nonumber
\end{equation}By Lemma 1.4, 1), it follows that
\begin{equation}
\frac{13}{21}<\psi'(G)=\psi'(G_1)\psi'(G_2)\cdots\psi'(G_k)\leq\psi'(G_i),\, \forall\, i=2,3,...,k.
\end{equation}If there is $i$ such that $G_i$ is not cyclic, then (1) gives
\begin{equation}
\psi'(G_i)<\frac{1}{p_i-1}\leq\frac{1}{2}\,,\nonumber
\end{equation}contradicting (2). So, we have
\begin{equation}
G\cong G_1\times C_m, \mbox{ where } m \mbox{ is odd},\nonumber
\end{equation}and
\begin{equation}
\frac{13}{21}<\psi'(G)=\psi'(G_1).\nonumber
\end{equation}This leads to
\begin{equation}
\psi(G_1)>\frac{13}{21}\,\psi(C_{2^{n_1}})\geq\frac{13}{21}\,\frac{2}{3}\,2^{2n_1}=\frac{26}{63}\,2^{2n_1}\nonumber
\end{equation}and so there exists $x\in G_1$ with $o(x)>\frac{26}{63}\,2^{n_1}$, i.e.
\begin{equation}
[G_1:\langle x\rangle]<\frac{63}{26}\,.\nonumber
\end{equation}Clearly, this implies that $[G_1:\langle x\rangle]=2$, i.e. $G_1$ possesses a cyclic maximal subgroup. Using  Theorem 4.1 of \cite{11}, II, we infer that either $G_1$ is abelian of type $C_2\times C_{2^{n_1-1}}$, $n_1\geq 2$, or non-abelian of one of the following types:
\begin{itemize}
\item[-] $M(2^{n_1})=\langle x,y\mid x^{2^{n_1-1}}=y^2=1\,, yxy=x^{2^{n_1-2}+1}\rangle$, $n_1\geq 4$;
\item[-] $D_{2^{n_1}}=\langle x,y\mid x^{2^{n_1-1}}=y^2=1\,, yxy=x^{-1}\rangle$;
\item[-] $Q_{2^{n_1}}=\langle x,y\mid x^{2^{n_1-1}}=y^4=1\,, yxy^{-1}=x^{2^{n_1-1}-1}\rangle$;
\item[-] $S_{2^{n_1}}=\langle x,y\mid x^{2^{n_1-1}}=y^2=1\,, yxy=x^{2^{n_1-2}-1}\rangle$, $n_1\geq 4$.
\end{itemize}If $G_1\cong C_2\times C_{2^{n_1-1}}$, then
\begin{equation}
\psi'(G_1)=\frac{2^{2n_1}+5}{2^{2n_1+1}+1}>\frac{13}{21}\Leftrightarrow 2^{2n_1}<\frac{92}{5}\Leftrightarrow n_1=2, \mbox{ i.e. } G_1\cong C_2\times C_2,\nonumber
\end{equation}while if $G_1$ is non-abelian, then we get:
\begin{itemize}
\item[-] $\psi'(M(2^{n_1}))=\displaystyle\frac{3\cdot 2^{n_1}}{2^{2n_1+1}+1}<\frac{13}{21}\,,\,\forall\, n_1\geq 4$;
\item[-] $\psi'(D_{2^{n_1}})=\displaystyle\frac{2^{2n_1-1}+3\cdot 2^{n_1}+1}{2^{2n_1+1}+1}<\frac{13}{21}\,,\,\forall\, n_1\geq 3$;
\item[-] $\psi'(Q_{2^{n_1}})=\displaystyle\frac{2^{2n_1-1}+3\cdot 2^{n_1+1}+1}{2^{2n_1+1}+1}>\frac{13}{21}\Leftrightarrow n_1=3$, i.e. $G_1\cong Q_8$;
\item[-] $\psi'(S_{2^{n_1}})=\displaystyle\frac{2^{2n_1-1}+9\cdot 2^{n_1-1}+1}{2^{2n_1+1}+1}<\frac{13}{21}\,,\,\forall\, n_1\geq 4$.
\end{itemize}This completes the proof.
\end{proof}

We also state an elementary lemma which will be useful to us in the sequel.

\begin{lemma}
Let $p$ be an odd prime and $P$ be a cyclic $p$-group of order $p^n$. Then
\begin{equation}
\frac{|P|}{\psi(C_{|P|})}\leq\frac{3}{7}\,,\nonumber
\end{equation}and the equality occurs if and only if $p=3$ and $n=1$.
\end{lemma}

We are now able to prove our main result.

\bigskip\noindent{\bf Proof of Theorem 1.1.} We will proceed by induction on $|G|$. If $G$ is cyclic, we are done. If $G$ is not cyclic, then $\psi'(G)>\frac{13}{21}$ and (1) lead to $q=2$. Also, we can assume that $p\geq 3$, i.e. $G$ is not a $2$-group. Then $G$ has a cyclic normal Sylow $r$-subgroup $P$, where either $r=2$ or $r=p$, by Lemma 2.1. Now Lemma 1.4, 4), implies that
\begin{equation}
\frac{13}{21}<\psi'(G)\leq\psi'(P)\psi'(G/P)=\psi'(G/P)\nonumber
\end{equation}and so $G/P$ is nilpotent by the inductive hypothesis.

If $r=2$, then $G$ has a normal $2$-complement $H$ and we infer that
\begin{equation}
G\cong P\times H\cong P\times(G/P)\nonumber
\end{equation}is nilpotent, as desired.

Next assume that $r=p$. Since $\psi'(G)>\frac{13}{21}>\frac{211}{1617}\,$, Theorem C shows that $G$ is solvable\footnote{See Theorem 6 of \cite{6} for an alternative argument.}, and consequently it has a $p$-complement $H$. Also, since $H\cong G/P$ is nilpotent and $\psi'(H)>\frac{13}{21}\,$, by Lemma 2.2 it follows that\newpage
\begin{equation}
H\cong H_1\times C_m,\nonumber
\end{equation}where $m$ is odd and $H_1$ is isomorphic with $Q_8$, $C_2\times C_2$ or $C_{2^{n_1}}$. On the other hand, by Lemma 1.4, 5), we get
\begin{equation}
\psi(G)=|P|\psi(H)+(\psi(P)-|P|)\psi(C_H(P))\nonumber
\end{equation}and so
\begin{equation}
\psi'(G)=\frac{|P|}{\psi(C_{|P|})}\,\psi'(H)+\left(1-\frac{|P|}{\psi(C_{|P|})}\right)\,\frac{\psi(C_H(P))}{\psi(C_{|H|})}\,.
\end{equation}

Obviously, if the semidirect product $G=P\rtimes H$ is trivial, then $G$ is nilpotent. In what follows, we will prove that if the semidirect product $G=P\rtimes H$ is non-trivial, then $\psi'(G)\leq\frac{13}{21}\,$, contradicting our hypothesis.

Since $C_H(P)$ is a proper subgroup of $H$, one obtains
\begin{equation}
\psi(C_H(P))\leq {\rm max}\{\psi(M)\mid M \!=\!\!\mbox{ maximal subgroup of } H\}.
\end{equation}By looking to the structure of maximal subgroups of $H$, we are able to compute the right side of (4). We distinguish the following three cases:
\begin{itemize}
\item[] Case 1. $H_1\cong Q_8$

\hspace{-12mm} We have $\psi'(H)=\frac{27}{43}$ and $\psi(C_H(P))\leq\psi(C_4\times C_m)$. Then
\begin{equation}
\frac{\psi(C_H(P))}{\psi(C_{|H|})}\leq\frac{11}{43}\nonumber
\end{equation}\hspace{-10mm}and (3) leads to
$$\hspace{-5mm}\psi'(G)\leq\frac{|P|}{\psi(C_{|P|})}\,\frac{27}{43}+\left(1-\frac{|P|}{\psi(C_{|P|})}\right)\frac{11}{43}=\frac{11}{43}\,+\frac{16}{43}\,\frac{|P|}{\psi(C_{|P|})}\leq\mbox{ (by Lemma 2.3)}$$
$$\hspace{-65mm}\leq\frac{11}{43}\,+\frac{16}{43}\,\frac{3}{7}\,= \frac{125}{301}<\frac{13}{21}\,.$$
\item[] Case 2. $H_1\cong C_2\times C_2$

\hspace{-12mm} We have $\psi'(H)=\frac{7}{11}$ and $\psi(C_H(P))\leq\psi(C_2\times C_m)$. Then
\begin{equation}
\frac{\psi(C_H(P))}{\psi(C_{|H|})}\leq\frac{3}{11}\nonumber
\end{equation}\hspace{-10mm}and (3) leads to
$$\hspace{-5mm}\psi'(G)\leq\frac{|P|}{\psi(C_{|P|})}\,\frac{7}{11}+\left(1-\frac{|P|}{\psi(C_{|P|})}\right)\frac{3}{11}=\frac{3}{11}\,+\frac{4}{11}\,\frac{|P|}{\psi(C_{|P|})}\leq\mbox{ (by Lemma 2.3)}$$
$$\hspace{-69mm}\leq\frac{3}{11}\,+\frac{4}{11}\,\frac{3}{7}\,= \frac{3}{7}<\frac{13}{21}\,.$$
\item[] Case 3. $H_1\cong C_{2^{n_1}}$

\hspace{-12mm} We have $\psi'(H)=1$ and $\psi(C_H(P))\leq\psi(C_{2^{n_1-1}}\times C_m)$. Then
\begin{equation}
\frac{\psi(C_H(P))}{\psi(C_{|H|})}\leq\frac{2^{2n_1-1}+1}{2^{2n_1+1}+1}\leq\!\footnote{Note that we have equality if and only if $n_1=1$.}\,\frac{1}{3}\nonumber
\end{equation}\hspace{-10mm}and (3) leads to
$$\hspace{-5mm}\psi'(G)\leq\frac{|P|}{\psi(C_{|P|})}\,1+\left(1-\frac{|P|}{\psi(C_{|P|})}\right)\frac{1}{3}=\frac{1}{3}\,+\frac{2}{3}\,\frac{|P|}{\psi(C_{|P|})}\leq\mbox{ (by Lemma 2.3)}$$
$$\hspace{-82mm}\leq\frac{1}{3}\,+\frac{2}{3}\,\frac{3}{7}\,= \frac{13}{21}\,.$$
\end{itemize}

Finally, we remark that we have $\psi'(G)=\frac{13}{21}$ if and only if $|P|=3$ and $n_1=1$, i.e.
\begin{equation}
G\cong C_3\rtimes(C_2\times C_m)\cong (C_3\rtimes C_2)\times C_m\cong S_3\times C_m \mbox{ with } (6,m)=1.\nonumber
\end{equation}The proof of Theorem 1.1 is now complete.$\qed$

\vspace*{3ex}\small

\hfill
\begin{minipage}[t]{5cm}
Marius T\u arn\u auceanu \\
Faculty of  Mathematics \\
``Al.I. Cuza'' University \\
Ia\c si, Romania \\
e-mail: {\tt tarnauc@uaic.ro}
\end{minipage}

\end{document}